\documentclass{amsart}
\usepackage{amsmath,amssymb,amsthm,amsfonts,amscd,mathrsfs,mathtools,enumerate}
\usepackage{tikz,graphicx}
\usepackage[all]{xy}
\usepackage[hyphens]{url}
\usepackage{hyperref}
\usepackage{caption}
\usetikzlibrary{decorations.markings}

\theoremstyle{plain}
    \newtheorem{thm}{Theorem}[section]
    \newtheorem{lem}[thm]   {Lemma}
    \newtheorem{cor}[thm]   {Corollary}
    
    \newtheorem*{thm*}{Theorem A}

\theoremstyle{definition}

    \newtheorem{rem}[thm]{Remark}

\def\Homeo{\operatorname{Homeo}}
\def\Mod{\operatorname{Mod}}
\def\Aut{\operatorname{Aut}}
\def\Out{\operatorname{Out}}

\newcommand{\be}{\begin{enumerate}}
\newcommand{\ee}{\end{enumerate}}
\newcommand{\R}{\mathbb{R}}
\newcommand{\Z}{\mathbb{Z}}

\begin{document}

\title[Closed surface braids]{Isotopy and homeomorphism of closed surface braids}

\author{Mark Grant}

\author{Agata Sienicka}

\address{Institute of Mathematics,
Fraser Noble Building,
University of Aberdeen,
Aberdeen AB24 3UE,
UK}

\address{Mathematical Institute, 
University of Bonn, 
Endenicher Allee 60,
D-53115 Bonn,
Germany}

\email{mark.grant@abdn.ac.uk}

\email{s6agsien@uni-bonn.de}

\date{\today}


\begin{abstract}
The closure of a braid in a closed orientable surface $\Sigma$ is a link in $\Sigma\times S^1$. We classify such closed surface braids up to isotopy and homeomorphism (with a small indeterminacy for isotopy of closed sphere braids), algebraically in terms of the surface braid group. We find that in positive genus, braids close to isotopic links if and only if they are conjugate, and close to homeomorphic links if and only if they are in the same orbit of the outer action of the mapping class group on the surface braid group modulo its center.
\end{abstract}


\maketitle
\section{Introduction}\label{sec:intro}

The relationship between braids and links has a long history, starting with Alexander \cite{Alexander23} who showed that every link in $3$-space is the closure of some braid. Markov's theorem \cite{Markov} asserts that two braids close to isotopic links in $3$-space if and only if they are related by a finite sequence of so-called \emph{Markov moves}. The first type of Markov move amounts to conjugation in some braid group $B_n$, while the second type, stabilisation, involves embedding $B_n$ into $B_{n+1}$ by adding an extra strand which crosses the last. Both of these results are proved in more detail in the book of Birman \cite{Bir}, and Morton \cite{Morton2} gave alternative proofs using threadings of diagrams.

The theorems of Alexander and Markov were generalised to links in arbitrary $3$-manifolds by Skora \cite{Skora}. Subsequently Sundheim gave alternative proofs of Skora's results using Morton's technique of threadings of diagrams \cite{Sundheim}. Both authors use the fact that every closed $3$-manifold $V$ has an open book decomposition with non-empty, connected binding $A$ and with page a connected, oriented compact surface $D$ with boundary. Then every oriented link in $V$ is isotopic to the closure of a braid in $D$ with respect to the braid axis $A$, and two links are isotopic if they differ by a finite sequence of moves which amount to conjugation in some surface braid group $B_n(D)$ and stabilisation with respect to certain maps $B_n(D)\to B_{n+1}(D)$. This second type of move, which  allows isotopies which alter the number of strands by passing a strand across $A$, relies for its existence on the non-vacuity of both $A$ and the boundary of $D$.

In this paper we consider the classification of links in $3$-manifolds with open book decompositions with \emph{empty} binding. The prototype is the following theorem of Artin \cite{Artin}, further proofs of which appear in the paper \cite{Morton} by Morton and in textbooks by Burde and Zieschang \cite{BZ} and Kassel and Turaev \cite{KT}. The solid torus $D^2\times S^1$ is an open book with empty binding and with $D^2$ as its page. Note that not every oriented link in the solid torus is isotopic to the closure of a disk braid $\beta\in B_n=B_n(D^2)$, since a closed $n$-braid must intersect each fibre $D^2\times\{\theta\}$ transversely in $n$ points. However, the isotopy classification of closed braids is simpler.

\begin{thm*}[{\cite{Artin,BZ,KT,Morton}}]\label{thm:A}
Two braids $\beta, \beta'\in B_n=B_n(D^2)$ close to isotopic links in $D^2\times S^1$ if and only if they are conjugate in $B_n$.
\end{thm*}

The notion of homeomorphism of links gives an \emph{a priori} weaker relation on links than that of isotopy. Two oriented links $L,L'\subset V$ are \emph{homeomorphic} if there exists a homeomorphism $h:V\to V$ which restricts to the identity on $\partial V$ and such that $h(L)=L'$ (preserving orientations). Theorem A holds also for homeomorphism, and thus there is no distinction between homeomorphism and isotopy for closed braids in the solid torus. This can be viewed as a consequence of the fact that any orientation preserving homeomorphism of the disk which restricts to the identity on the boundary is isotopic to the identity.

For the rest of the paper $\Sigma$ will denote a closed orientable surface of genus $g$. An element $\beta$ of the surface braid group $B_n(\Sigma)$ may be represented as a geometric braid in $\Sigma\times I$, and by identifying the two ends of this cylinder we obtain its closure $\hat\beta$, an oriented link in the closed $3$-manifold $\Sigma\times S^1$. In this paper we investigate the classification of such closed surface braids in $\Sigma\times S^1$ up to both isotopy and homeomorphism, stating our results in terms of the algebra of surface braid groups.

\begin{thm}\label{thm:genus2isotopy}
If $g\ge2$, then two surface braids $\beta,\beta'\in B_n(\Sigma)$ close to isotopic links in $\Sigma\times S^1$ if and only if they are conjugate in $B_n(\Sigma)$.
\end{thm}

One of the main tools in proving Theorem \ref{thm:genus2isotopy} is the so called \emph{Birman exact sequence}
\[
\xymatrix{
1 \ar[r] & B_n(\Sigma) \ar[r] & \Mod(\Sigma,Q) \ar[r] & \Mod(\Sigma) \ar[r] & 1,
}
\]
the details of which are recalled in Section 2 below. This sequence furnishes an outer action $\Mod(\Sigma)\to \Out(B_n(\Sigma))$ of the mapping class group of $\Sigma$ on its braid group. Therefore the mapping class group acts on the set of conjugacy classes in $B_n(\Sigma)$.

\begin{thm}\label{thm:genus2homeomorphism}
If $g\ge2$ then two surface braids $\beta,\beta'\in B_n(\Sigma)$ close to homeomorphic links in $\Sigma\times S^1$ if and only if their conjugacy classes $[\beta]$ and $[\beta']$ are in the same orbit under the outer action of $\Mod(\Sigma)$ on $B_n(\Sigma)$.
\end{thm}

Thus the difference between homeomorphism and isotopy in the higher genus case is precisely measured by the potential non-triviality of the outer action of $\Mod(\Sigma)$ on $B_n(\Sigma)$. We note that $\Out(B_n(\Sigma))$ is isomorphic to the extended mapping class group $\Mod^\pm(\Sigma)$, see \cite{Bellingeri,Zhang,An}. The action of $\Mod^\pm(\Sigma)$ on $B_n(\Sigma)$ has been intensively studied, see for instance \cite{KidaYamagata}.

In the lower genus cases where either $g=0$ and $n\ge3$ or $g=1$ and $n\ge2$, then as described in \cite[Theorem 4.3]{Bir} there is an exact sequence
\[
 \xymatrix{
1 \ar[r] & \mathcal{Z}B_n(\Sigma) \ar[r] &  B_n(\Sigma) \ar[r] & \Mod(\Sigma,Q) \ar[r] & \Mod(\Sigma) \ar[r] & 1,
}
\]
where $\mathcal{Z}B_n(\Sigma)$ denotes the centre of the surface braid group. We obtain the following result for sphere braids.

\begin{thm}\label{thm:genus0}
Let $n\ge3$.

\begin{enumerate}[(a)]
\item Two sphere braids $\beta,\beta'\in B_n(S^2)$ close to homeomorphic links in $S^2\times S^1$ if and only if they are conjugate in $B_n(S^2)$ up to multiplication by the non-trivial element $\Delta^2\in \mathcal{Z}B_n(S^2)$.
\item Consequently, if two sphere braids $\beta,\beta'\in B_n(S^2)$ close to isotopic links in $S^2\times S^1$, then they are conjugate in $B_n(S^2)$ up to multiplication by the non-trivial element $\Delta^2\in \mathcal{Z}B_n(S^2)$.
\end{enumerate}
\end{thm}

\begin{rem}
The non-trivial element $\Delta^2$ of $\mathcal{Z}B_n(S^2)\cong \Z/2$ is represented by the \emph{full twist}, a rotation of $n$ points spaced along a single line of longitude through $2\pi$ radians about an axis through the poles (see \cite[Lemma 4.2.3]{Bir}, for example). Its closure $\widehat{\Delta^2}$ is homeomorphic to the closure of the trivial braid, the required homeomorphism of $S^2\times S^1$ being given by a Gluck twist \cite{Gluck}. Since the Gluck twist is not isotopic to the identity, one conjectures that $\widehat{\Delta^2}$ is not isotopic to the closure of the trivial braid, which would imply that two sphere braids have isotopic closures if and only if they are conjugate. This appears to require new methods for distinguishing non-isotopic links in $3$-manifolds. We understand that a proof will appear in a forthcoming paper by Paolo Aceto, Corey Bregman, Christopher W. Davis, JungHwan Park, and Arunima Ray with the working title ``Automorphisms of 3-manifolds acting on knots".
\end{rem}

In the torus case, we state our result in terms of the extension
\[
 \xymatrix{
1 \ar[r] & B_n(T^2)/\mathcal{Z}B_n(T^2) \ar[r] & \Mod(T^2,Q) \ar[r] & \Mod(T^2) \ar[r] & 1,
}
\]
and the resulting outer action of $\Mod(T^2)\cong \operatorname{SL}(2,\Z)$ on the group $B_n(T^2)/\mathcal{Z}B_n(T^2)$.

\begin{thm}\label{thm:genus1}
Let $n\ge2$.

\begin{enumerate}[(a)]
\item Two torus braids $\beta,\beta'\in B_n(T^2)$ close to homeomorphic links in $T^2\times S^1$ if and only if their representatives in $B_n(T^2)/\mathcal{Z}B_n(T^2)$ are in the same orbit under the outer action of $\Mod(T^2)$.
\item Two torus braids $\beta,\beta'\in B_n(T^2)$ close to isotopic links in $T^2\times S^1$ if and only if they are conjugate in $B_n(T^2)$.
\end{enumerate}
\end{thm}

Our proofs are based on the proof of Theorem A given by Kassel and Turaev in \cite[Chapter 2]{KT}. For braids in the disk the main ingredient is the faithful representation $B_n\hookrightarrow\operatorname{Aut}(F_n)$, and the description of its image given by Artin \cite{Artin}. For surface braids the main ingredients are the Birman exact sequence and the Dehn--Nielsen--Baer theorem (for closed surfaces with and without punctures), the statements of which are recalled in Sections 2 and 3 respectively. The proofs of Theorems \ref{thm:genus2isotopy} and \ref{thm:genus2homeomorphism} are given in Section 4, while the proofs of Theorems \ref{thm:genus0} and \ref{thm:genus1} are given in Section 5.

While the isotopy classifications in \cite{Skora} and \cite{Sundheim} appear to be more widely applicable, some care would be needed to deduce our Theorems \ref{thm:genus2isotopy}, \ref{thm:genus0}(b) and \ref{thm:genus1}(b) as specializations, since the proofs assume an open book decomposition with non-empty binding. For example, the proof of \cite[Theorem 4.1]{Skora} uses faithfulness of the representation of the mapping class group of a punctured surface on its fundamental group, which involves choosing a base-point in a boundary component, and is not true as stated for sphere braids. We prefer to regard these results as generalisations of Theorem A. By generalising the proof in \cite{KT}, we obviate the need to work with specific combinatorial deformations, or the use of Cerf theory as in \cite[Lemma 6.1]{Sundheim}. Our methods also lead easily to the results on homeomorphism, which appear to be new.

Finally we mention an immediate corollary of our Theorems \ref{thm:genus2isotopy} and \ref{thm:genus1}. Let us say that two closed braids are \emph{transversely isotopic} if they are isotopic through closed braids.

\begin{cor}
For closed surface braids in $\Sigma\times S^1$ in genus $g\ge1$, isotopy implies transverse isotopy.
\end{cor}

This research was carried out while the second author held a London Mathematical Society Undergraduate Bursary at the University of Aberdeen, under the supervision of the first author. Both authors would like to thank the London Mathematical Society for their financial support. We also thank Paolo Bellingeri, Tyrone Ghaswala, Richard Hepworth, Andrew Lobb, Mark Powell and Oliver Singh for useful conversations, and the anonymous referee for their helpful comments.

\section{The Birman exact sequence}

In this section we recall some well known results in order to fix our notation. Let $\Sigma=\Sigma_g$ be a closed orientable surface of genus $g$, and let $Q=\{z_1,\ldots , z_n\}$ be a set of $n$ disjoint marked points of $\Sigma$. Let $F_n(\Sigma)=\{(x_1,\ldots , x_n)\in \Sigma^n \mid x_i\neq x_j\mbox{ for }i\neq j\}$ be the \emph{$n$ point ordered configuration space} of $\Sigma$. This space carries a natural action of the symmetric group $\mathfrak{S}_n$ whose orbit space $C_n(\Sigma):=F_n(\Sigma)/\mathfrak{S}_n$ is the \emph{$n$ point unordered configuration space} of $\Sigma$. It is well known (or may be taken as a definition) that $\pi_1(C_n(\Sigma);Q)\cong B_n(\Sigma)$, the \emph{$n$ strand surface braid group} of $\Sigma$. Here we have taken as base point the set $Q$ mentioned above (but note that $C_n(\Sigma)$ is path-connected).

Now let $\Homeo^+(\Sigma)$ denote the topological group of orientation preserving self-homeomorphisms of $\Sigma$ in the compact-open topology. There is a map
\begin{equation}\label{eq:fibration}
\Homeo^+(\Sigma)\to C_n(\Sigma),\qquad h\mapsto h(Q)
\end{equation}
 which turns out to be a locally trivial fibration (see \cite[Lemma 1.35]{KT} for example). Its fiber is the space $\Homeo^+(\Sigma,Q)$ of orientation preserving self-homeomorphisms of $\Sigma$ which fix $Q$ set-wise (the points of $Q$ may be permuted).

Both $\Homeo^+(\Sigma)$ and $\Homeo^+(\Sigma,Q)$ are topological groups under composition of homeomorphisms. The group of path components $\pi_0(\Homeo^+(\Sigma))$ is denoted $\Mod(\Sigma)$, and called the \emph{mapping class group} of $\Sigma$. The group of path components $\pi_0(\Homeo^+(\Sigma,Q))$ is denoted $\Mod(\Sigma,Q)$, and is the \emph{mapping class group of $\Sigma$ with marked points $Q$}. By examining the lower portion of the long exact homotopy sequence of fibration (\ref{eq:fibration}), one may establish the following result.

\begin{thm}[{See \cite[Theorem 4.3]{Bir}}]\label{thm:Birman}
If $g\ge2$, there is a short exact sequence of groups
\[
\xymatrix{
1 \ar[r] & B_n(\Sigma) \ar[r]^-\partial & \Mod(\Sigma,Q) \ar[r]^-{j_*} & \Mod(\Sigma) \ar[r] & 1.
}
\]

If $g=0$ and $n\ge3$, or $g=1$ and $n\ge2$, there is an exact sequence of groups
\[
\xymatrix{
1 \ar[r] & \mathcal{Z}B_n(\Sigma) \ar[r] & B_n(\Sigma) \ar[r]^-\partial & \Mod(\Sigma,Q) \ar[r]^-{j_*} & \Mod(\Sigma) \ar[r] & 1,
}
\]
where $\mathcal{Z}B_n(\Sigma)$ denotes the centre of $B_n(\Sigma)$.
\end{thm}

The maps in the above sequences can be described as follows. The homomorphism $j_*:\Mod(\Sigma,Q) \to \Mod(\Sigma)$ is induced by the obvious inclusion $j:\Homeo^+(\Sigma,Q)\hookrightarrow\Homeo^+(\Sigma)$. The homomorphism $\partial:B_n(\Sigma)=\pi_1(C_n(\Sigma);Q)\to \pi_0(\Homeo^+(\Sigma,Q))=\Mod(\Sigma,Q)$ is the connecting morphism of the fibration (\ref{eq:fibration}), and has the following geometric description. A braid $\beta$ in $B_n(\Sigma)$ may be represented by an isotopy $\beta_t:Q\to \Sigma$ of the set of marked points $Q\subset \Sigma$, which by the Isotopy Extension theorem extends to an isotopy $\overline{\beta}_t:\Sigma\to \Sigma$ of the whole surface such that $\overline{\beta}_0=\mathrm{id}_\Sigma$ and $\overline{\beta}_1(Q)=Q$. Then $\partial(\beta)= [\overline{\beta}_1]$. Because of this geometric description, the map $\partial$ is sometimes called the \emph{point-pushing map}.

Finally in this section we observe that we can replace the group $\Mod(\Sigma,Q)$ in the exact sequences of Theorem \ref{thm:Birman} with the group $\Mod(\Sigma-Q)=\pi_0(\Homeo^+(\Sigma-Q))$, the \emph{mapping class group of the punctured surface $\Sigma-Q$}. This is because a self-homeomorphism of the punctured surface $\Sigma-Q$ extends uniquely to a self-homeomorphism of $\Sigma$ which fixes $Q$ set-wise, via Freudenthal compactification.

\section{The Dehn--Nielsen--Baer theorem}

The Dehn--Nielsen--Baer theorem says roughly that mapping classes of surfaces can be understood entirely algebraically in terms of their action on the fundamental group. We require the version for closed surfaces as well as the version for surfaces with punctures. Modern references are the books of Farb and Margalit \cite{Prim}, Ivanov \cite{Ivan} and Zieschang, Vogt and Coldewey \cite{ZVC}.

For a closed orientable surface $\Sigma$, we denote by $\Homeo(\Sigma)$ the topological group of \emph{all} self-homeomorphisms of $\Sigma$ (not necessarily orientation preserving), and by $\Mod^{\pm}(\Sigma)=\pi_0(\Homeo(\Sigma))$ the \emph{extended mapping class group} of $\Sigma$. Any homeomorphism $h:\Sigma\to \Sigma$ induces an isomorphism $h_*:\pi_1(\Sigma;x)\to \pi_1(\Sigma;h(x))$ for a chosen base point $x$. Composing with a change of base point isomorphism induced by a path from $h(x)$ to $x$, we obtain an automorphism of $\pi_1(\Sigma;x)$. Choosing a different path may alter this automorphism, but only by an inner automorphism. Therefore $h_*$ represents a well-defined element of $\Out(\pi_1(\Sigma;x))$, the group of outer automorphisms of the fundamental group. (From now on we suppress the arbitrarily chosen base point from the notation, unless it becomes relevant.) In this way we obtain a homomorphism $\phi: \Mod^{\pm}(\Sigma)\to \Out(\pi_1(\Sigma))$.

\begin{thm}[Dehn--Nielsen--Baer for closed surfaces]\label{thm:DNBclosed}
When $g\ge1$, the homomorphism $\phi:\Mod^{\pm}(\Sigma)\to \Out(\pi_1(\Sigma))$ is an isomorphism.

When $g=0$, one has $\Mod^\pm(S^2)\cong\Z/2$ and $\Out(\pi_1(S^2))$ is trivial.
\end{thm}

The mapping class group $\Mod(\Sigma)$ sits inside the extended mapping class group $\Mod^\pm(\Sigma)$ as an index $2$ normal subgroup. When $g\ge1$ the surface $\Sigma$ is an Eilenberg--Mac Lane space $K(\pi_1(\Sigma),1)$, and the image of $\Mod(\Sigma)$ in $\Out(\pi_1(\Sigma))$ consists of those outer automorphisms which act trivially on $H_2(\pi_1(\Sigma);\Z)\cong\Z$. When $g=0$ the mapping class group $\Mod(S^2)$ is trivial.

We now state the Dehn--Nielsen--Baer Theorem for punctured surfaces. There is a homomorphism $\phi:\Mod^\pm(\Sigma-Q)\to \Out(\pi_1(\Sigma-Q))$, defined just as before. Its image must lie in the subgroup $\Out^*(\pi_1(\Sigma-Q))$ of $\Out(\pi_1(\Sigma-Q))$ consisting of outer automorphisms which preserve the set of $n$ conjugacy classes represented by simple closed curves surrounding the individual punctures. When $\Sigma-Q$ is hyperbolic, this follows from the fact that an automorphism of $\pi_1(\Sigma-Q)$ induced by a homeomorphism must preserve parabolic elements.

\begin{thm}[Dehn--Nielsen--Baer for punctured surfaces]\label{thm:DNBpunctures}
Suppose that $\Sigma-Q$ is hyperbolic (for instance, if its Euler characteristic $\chi(\Sigma-Q)$ is negative). Then the homomorphism
\[
\phi:\Mod^{\pm}(\Sigma-Q)\to \Out^*(\pi_1(\Sigma-Q))
\]
is an isomorphism.
\end{thm}

Again, $\Mod(\Sigma-Q)$ sits inside $\Mod^\pm(\Sigma-Q)$ as an index $2$ subgroup. When $n\ge1$ its image in $\Out^*(\pi_1(\Sigma-Q))$ consists of outer automorphism which preserve the set of conjugacy classes of \emph{oriented} simple closed curves around the punctures, where the orientations are induced by a fixed orientation of $\Sigma$.

Finally in this section we prove a naturality result relating Theorems \ref{thm:DNBclosed} and \ref{thm:DNBpunctures}. We begin by defining algebraically a homomorphism $J:\Out^*(\pi_1(\Sigma-Q))\to \Out(\pi_1(\Sigma))$. The group $\pi_1(\Sigma-Q)$ may be identified with the group $G$ with presentation
\[
G=\langle a_1,b_1,\ldots, a_g,b_g,d_1,\ldots , d_n \mid [a_1,b_1]\cdots [a_g,b_g]=d_1\cdots d_n\rangle,
\]
where the $d_i$ are classes of loops which go once around the individual punctures $z_i\in Q$. The homomorphism $\pi_1(\Sigma-Q)\to \pi_1(\Sigma)$ induced by inclusion is an epimorphism with kernel the normal closure $N:=\overline{\langle d_1,\ldots , d_n\rangle}$ of the subgroup generated by the $d_i$. Let $f\in \Aut(\pi_1(\Sigma-Q))$ be an automorphism representing $[f]\in \Out^*(\pi_1(\Sigma-Q))$. Then due to the condition that $[f]$ permutes the conjugacy classes of the $d_i$ and their inverses, we see that $f$ must preserve $N$. Therefore $f$ induces an automorphism $\bar{f}$ of $G/N\cong \pi_1(\Sigma)$. Conjugating $f$ by an element $g\in G$ has the effect of conjugating $\bar{f}$ by the image $\bar{g}\in G/N$. Hence the assignment $[f]\mapsto [\bar{f}]$ is well-defined, and is easily seen to be a homomorphism.

\begin{lem}\label{lem:DNBnat}
The following diagram commutes, for all $g\ge0$ and $n\ge1$:
\[
\xymatrix{
\Mod(\Sigma,Q) \ar[d]^\cong \ar[r]^-{j_*} & \Mod(\Sigma) \ar@{^{(}->}[dd]^\phi \\
\Mod(\Sigma-Q) \ar@{^{(}->}[d]^\phi & & \\
\Out^*(\pi_1(\Sigma-Q)) \ar[r]^-{J} & \Out(\pi_1(\Sigma)).
}
\]
\end{lem}

\begin{proof}
Let $[h]\in \Mod(\Sigma,Q)$ be a mapping class represented by a homeomorphism of pairs $h:(\Sigma,Q)\to (\Sigma,Q)$. Letting $h|:\Sigma-Q\to\Sigma-Q$ be the restriction, we see that commutativity of the diagram amounts to the statement that $J[h|_*]=[h_*]\in\Out(\pi_1(\Sigma))$. This follows from commutativity of the diagram
\[
\xymatrix{
\pi_1(\Sigma-Q;x) \ar[d] \ar[r]^-{h|_*} & \pi_1(\Sigma-Q;h(x)) \ar[d] \\
\pi_1(\Sigma;x) \ar[r]^-{h_*} & \pi_1(\Sigma;h(x))
}
\]
which shows that $h_*$ is the automorphism induced by $h|_*$.
\end{proof}

\section{Proofs in higher genus}

Let $\Sigma$ be a closed orientable surface, and let $I=[0,1]$ denote the closed unit interval. An element of the surface braid group $B_n(\Sigma)=\pi_1(C_n(\Sigma);Q)$ may be represented by a geometric braid $\beta$ in $\Sigma\times I$, that is, an isotopy class of $n$ non-intersecting arcs in $\Sigma\times I$ connecting the points $Q\times\{0\}$ to the points $Q\times\{1\}$, and such that the projection of each arc to $I$ is a homeomorphism. On identifying the two ends $\Sigma\times\{0\}$ and $\Sigma\times\{1\}$ via $(x,0)\sim(x,1)$, we obtain a link $\hat\beta$ in the $3$-manifold $V:=\Sigma\times S^1$, called the \emph{closure} of $\beta$. Note that the closure of an $n$-strand surface braid intersects $\Sigma\times\{\theta\}\subset V$ transversely in $n$ points for every $\theta\in S^1$, and hence not every link in $V$ is isotopic to the closure of a surface braid.

Let $\overline{V}=\Sigma\times\R$. Multiplying $\Sigma$ by the infinite cyclic covering $\R\to \R/\Z=S^1$, we obtain a covering $p:\overline{V}\to V$ given by $p(x,t)=(x, [t])$. Denote by $T:\overline{V}\to\overline{V}$ the deck transformation $T(x,t)=(x,t-1)$.

Given a braid $\beta\in B_n(\Sigma)$, we denote its closure by $L\subset V$, and let $\overline{L}=p^{-1}(L)\subset \overline{V}$ denote the pre-image of the closure in the covering. Thus $\overline{L}$ consists of infinitely many copies of the braid $\beta$ stacked on top of one another. Such infinite braids can always be unravelled by pushing the braiding off to infinity. More precisely, let $\overline{\beta}_t:\Sigma\to \Sigma$ be an isotopy associated to the braid $\beta$, as in the description of the point-pushing map following Theorem \ref{thm:Birman}. Then the map
\[
H_\beta:\overline{V}\to \overline{V},\qquad (x,t)\mapsto (\overline{\beta}_{t-\lfloor t\rfloor}\overline{\beta}_1^{\lfloor t\rfloor}(x),t)
\]
is a level-preserving homeomorphism which throws $Q\times\R$ onto $\overline{L}$, and therefore induces a level-preserving homeomorphism $H_\beta: (\Sigma-Q)\times\R\to \overline{V}-\overline{L}$ (here and elsewhere below we use the same symbol to denote two homeomorphisms, one of which is the restriction of the other).

Let $i:\Sigma \hookrightarrow \overline{V}$ and its restriction $i:\Sigma-Q\hookrightarrow \overline{V}-\overline{L}$ be the inclusions given by $x\mapsto (x,0)$. Both are homotopy equivalences, with explicit homotopy inverses given by $r={\rm pr}_\Sigma\circ H_\beta^{-1}:\overline{V}\to \Sigma$ and its restriction
\[
\xymatrix{
r: \overline{V}-\overline{L} \ar[rr]^{H_\beta^{-1}} && (\Sigma-Q)\times\R \ar[r]^-{{\rm pr}_{\Sigma-Q}} & \Sigma-Q.
}
\]
Moreover, we have the following commutative diagram
\begin{equation}\label{eq:b}
\xymatrix{
\Sigma-Q \ar[r]^i \ar[d]_{\overline{\beta}_1} & \overline{V}-\overline{L} \ar[d]^T\\
\Sigma-Q & \overline{V}-\overline{L} \ar[l]^r
}
\end{equation}
relating the punctured surface homeomorphism associated to $\beta$ with the deck transformation $T$.

Similarly, given a second braid $\beta'$ with closure $L'\subset V$ and $\overline{L'}=p^{-1}(L')\subset \overline{V}$, we get a homotopy inverse $r':\overline{V}-\overline{L'}\to \Sigma-Q$ to the inclusion $i':\Sigma-Q\hookrightarrow\overline{V}-\overline{L'}$ and a commuting diagram
\begin{equation}\label{eq:b'}
\xymatrix{
\Sigma-Q \ar[r]^{i'} \ar[d]_{\overline{\beta'}_1} & \overline{V}-\overline{L'} \ar[d]^T\\
\Sigma-Q & \overline{V}-\overline{L'} \ar[l]^{r'}.
}
\end{equation}

We are now in a position to give proofs of Theorems \ref{thm:genus2isotopy} and \ref{thm:genus2homeomorphism} stated in the Introduction.

\begin{proof}[Proof of Theorem \ref{thm:genus2isotopy}]

It is clear that conjugate braids close to isotopic links.

Conversely, suppose that $\beta,\beta'\in B_n(\Sigma)$ have isotopic closures $L,L'\subset V$. Thus there is an isotopy $h_t:V\to V$ such that $h_0={\rm id}$ and $h_1(L) = L'$. An exercise in covering space theory shows that $h_t$ lifts to an isotopy $\overline{h}_t:\overline{V}\to \overline{V}$ with $\overline{h}_0={\rm id}$ and $\overline{h}_1(\overline{L})=\overline{L'}$ which commutes with the deck transformations, in that $\overline{h}_t \circ T = T \circ \overline{h}_t$.

Using $\overline{h}_1$ we may define a map $\Phi_h:\Sigma\to \Sigma$ and its restriction $\Phi_h:\Sigma-Q\to\Sigma-Q$ by commutativity of the following diagrams:

\begin{equation}\label{eq:h}
\xymatrix{
\Sigma \ar[r]^-{i} \ar[d]_{\Phi_h} & \overline{V} \ar[d]^{\overline{h}_1} & & \Sigma-Q \ar[r]^-{i} \ar[d]_{\Phi_h} & \overline{V}-\overline{L} \ar[d]^{\overline{h}_1} \\
\Sigma & \overline{V}\ar[l]^-{r'} & & \Sigma-Q & \overline{V}-\overline{L'} \ar[l]^-{r'}.
}
\end{equation}

Note that both $\Phi_h$ and its restriction are homotopy equivalences (with homotopy inverse $r\circ \overline{h}_1^{-1}\circ i'$), and that furthermore $\Phi_h:\Sigma\to \Sigma$ is homotopic to the identity.

Next, by Theorem \ref{thm:Birman} and Lemma \ref{lem:DNBnat} we have a commutative diagram whose top row is exact:
\begin{equation}\label{eq:strategy}
\xymatrix{
1 \ar[r] & B_n(\Sigma) \ar@{^{(}->}[ddr]_{\iota} \ar[r]^-{\partial} & \Mod(\Sigma,Q) \ar[d]^\cong \ar[r]^-{j_*} & \Mod(\Sigma) \ar@{^{(}->}[dd]^\phi \ar[r] & 1 \\
&& \Mod(\Sigma-Q) \ar@{^{(}->}[d]^\phi & & \\
 && \Out^*(\pi_1(\Sigma-Q)) \ar[r]^-{J} & \Out(\pi_1(\Sigma)).
}
\end{equation}
The embedding $\iota :B_n(\Sigma)\hookrightarrow \Out^*(\pi_1(\Sigma-Q))$ satisfies $\iota(\beta) =[(\overline{\beta}_1)_*]$ and $\iota(\beta') =[(\overline{\beta'}_1)_*]$. From diagrams (\ref{eq:b}) and (\ref{eq:b'}) and the equivariance of $\overline{h}_1$ we obtain the larger commutative diagram
\begin{equation}\label{eq:conj}
\xymatrix{
\Sigma-Q \ar[r]^i \ar[d]_{\overline{\beta}_1} & \overline{V}-\overline{L} \ar[d]^T \ar[r]^{\overline{h}_1} & \overline{V}-\overline{L'} \ar[r]^{r'} \ar[d]^{T} & \Sigma-Q \ar[d]^{\overline{\beta'}_1}\\
\Sigma-Q & \overline{V}-\overline{L} \ar[l]^r &  \overline{V}-\overline{L'} \ar[l]^{\overline{h}_1^{-1}} & \Sigma-Q \ar[l]^{i'}
}
\end{equation}
whose top row is $\Phi_h$ and whose bottom row (read right to left) is its homotopy inverse. This shows that $[(\Phi_h)_*]$ conjugates $\iota(\beta)$ to $\iota(\beta')$ in the group $\Out^*(\pi_1(\Sigma-Q))$. It only remains to note that $[(\Phi_h)_*]=\iota(\alpha)$ for some braid $\alpha\in B_n(\Sigma)$; this follows from the Dehn--Nielsen--Baer theorem, together with Lemma \ref{lem:DNBnat} and the fact that $\Phi_h:\Sigma\to \Sigma$ induces the identity outermorphism.
\end{proof}

\begin{proof}[Proof of Theorem \ref{thm:genus2homeomorphism}]
Suppose that $\beta,\beta'\in B_n(\Sigma)$ are in the same orbit under the outer action of $\Mod(\Sigma)$ induced by the extension
\[
\xymatrix{
1 \ar[r] & B_n(\Sigma) \ar[r]^-{\partial} & \Mod(\Sigma,Q) \ar[r] & \Mod(\Sigma) \ar[r] & 1.
}
\]
This amounts to saying that $\partial(\beta)$ and $\partial(\beta')$ are conjugate in $\Mod(\Sigma,Q)$. Hence there is a homeomorphism $h:(\Sigma,Q)\to (\Sigma,Q)$ such that $h\overline{\beta}_1 h^{-1} = \overline{\beta'}_1$. We then define a homeomorphism
\[
H:V\to V,\qquad [x,t]\mapsto [\overline{\beta'}_t\circ h\circ \overline{\beta}_t^{-1}(x),t],
\]
which by direct verification has $H(\hat\beta)=\hat{\beta'}$ (note that it may permute the components, according as $h$ permutes the points of $Q$).

Conversely, suppose that $\beta,\beta'\in B_n(\Sigma)$ have homeomorphic closures $L,L'\subset V$. Let $h:V \to V$ be an orientation preserving homeomorphism such that $h(L)=L'$. Exactly as in the proof of Theorem \ref{thm:genus2isotopy}, one may use $h$ to obtain homotopy equivalences $\Phi_h:\Sigma\to \Sigma$ and $\Phi_h:\Sigma-Q\to\Sigma-Q$, such that $[(\Phi_h)_*]$ conjugates $\iota(\beta)$ into $\iota(\beta')$ in the group $\Out^*(\pi_1(\Sigma-Q))$. The only difference is that now $\Phi_h:\Sigma\to \Sigma$ need not be homotopic to the identity, and so we may only conclude that $\partial(\beta)$ and $\partial(\beta')$ are conjugate in $\Mod(\Sigma,Q)$. The result follows.
\end{proof}

\section{Proofs in low genus}

When the genus is $0$ or $1$, the proofs proceed similarly, except that the diagram (\ref{eq:strategy}) is replaced by the diagram
\begin{equation}\label{eq:strategy2}
\xymatrix{
1 \ar[r] & \mathcal{Z}B_n(\Sigma) \ar[r] &  B_n(\Sigma) \ar@{^{(}->}[ddr]_{\iota} \ar[r]^-{\partial} & \Mod(\Sigma,Q) \ar[d]^\cong \ar[r]^-{j_*} & \Mod(\Sigma) \ar@{^{(}->}[dd]^\phi \ar[r] & 1 \\
&&& \Mod(\Sigma-Q) \ar@{^{(}->}[d]^\phi & & \\
 &&& \Out^*(\pi_1(\Sigma-Q)) \ar[r]^-{J} & \Out(\pi_1(\Sigma)).
}
\end{equation}

\begin{proof}[Proof of Theorem \ref{thm:genus0}]
When $n\ge3$ the center $\mathcal{Z}B_n(S^2)$ is of order $2$, generated by the full twist $\Delta^2\in B_n(S^2)$ (see \cite{GilletteVanB}). The mapping class group $\Mod(S^2)$ is trivial.
\begin{enumerate}[(a)]
\item Suppose $\beta,\beta'\in B_n(S^2)$ close to homeomorphic links in $S^2\times S^1$. Repeating the argument used to prove Theorem \ref{thm:genus2homeomorphism}, we see that $\partial(\beta)$ and $\partial(\beta')$ are conjugate in $\Mod(S^2,Q)$. Since $\Mod(S^2)$ is trivial, there is a short exact sequence
    \[
    \xymatrix{
    1 \ar[r] &  \mathcal{Z}B_n(S^2) \ar[r] &  B_n(S^2) \ar[r]^-{\partial} & \Mod(S^2,Q) \ar[r] & 1
    }
    \]
    from which it follows that $\beta$ is conjugate to $\beta'$ or to $\Delta^2\beta'$.

    To prove the converse, it suffices to note that conjugate braids close to homeomorphic links, and that the closure $\widehat{\Delta^2}$ of the central element $\Delta^2$ is homeomorphic to the closure of the trivial braid, from which it follows that $\widehat{\beta\Delta^2}$ is homeomorphic to $\hat\beta$ for any braid $\beta$.
\item Suppose $\beta,\beta'\in B_n(S^2)$ close to isotopic links in $S^2\times S^1$. Again, an argument identical to the one used in Theorem \ref{thm:genus2isotopy} shows that $\partial(\beta)$ and $\partial(\beta')$ are conjugate in $\Mod(S^2,Q)$, from which it follows that $\beta$ and $\beta'$ are conjugate in $B_n(S^2)$, up to multiplication by $\Delta^2$.
    
    Alternatively, this follows directly from (a), since isotopic links are homeomorphic.
\end{enumerate}
\end{proof}

\begin{proof}[Proof of Theorem \ref{thm:genus1}]
When $n\ge2$ the center $\mathcal{Z} B_n(T^2)$ is free abelian of rank $2$, generated by full twists around each circle factor of $T^2=S^1\times S^1$ (see \cite{Bir,ParisRolfsen}).
\begin{enumerate}[(a)]
\item Suppose $\beta,\beta'\in B_n(T^2)$ close to homeomorphic links in $T^2\times S^1$. Repeating the previous arguments gives that $\partial(\beta)$ and $\partial(\beta')$ are conjugate in $\Mod(T^2,Q)$. Thus from the exact sequence
    \[
 \xymatrix{
1 \ar[r] & B_n(T^2)/\mathcal{Z}B_n(T^2) \ar[r]^-\partial & \Mod(T^2,Q) \ar[r] & \Mod(T^2) \ar[r] & 1,
}
\]
we see that the classes $[\beta],[\beta']\in B_n(T^2)/\mathcal{Z}B_n(T^2)$ are in the same orbit under the outer action of $\Mod(T^2)$, as claimed.

Conversely, if $[\beta],[\beta']\in B_n(T^2)/\mathcal{Z}B_n(T^2)$ are in the same orbit under the outer action of $\Mod(T^2)$, then $\partial(\beta)$ and $\partial(\beta')$ are conjugate in $\Mod(T^2,Q)$. An argument as in the proof of Theorem \ref{thm:genus2homeomorphism} then gives that the closures $\hat\beta$ and $\hat{\beta'}$ are homeomorphic links in $T^2\times S^1$.
\item Suppose $\beta,\beta'\in B_n(T^2)$ close to isotopic links in $T^2\times S^1$. One finds as in the proof of Theorem \ref{thm:genus0}(b) that they are conjugate in $B_n(T^2)$ up to multiplication by elements of $\mathcal{Z}B_n(T^2)$. However, multiplying a braid by a non-trivial element of $\mathcal{Z}B_n(T^2)$ will change the homology class of the resulting link in $T^2\times S^1$, resulting in a non-isotopic link. We therefore may conclude in this case that $\beta$ and $\beta'$ are conjugate in $B_n(T^2)$.
\end{enumerate}
\end{proof}

\end{document}